\documentclass[12pt]{amsart}
\usepackage{a4wide}
\usepackage[utf8]{inputenc}
\usepackage{amsmath}
\usepackage{amssymb}
\usepackage{amsfonts}
\usepackage{amsopn}
\usepackage{graphicx}
\usepackage{enumerate}
\usepackage{color}
\usepackage{mathtools}
\usepackage{mathrsfs}
\usepackage{bbm}
\usepackage{yhmath}
\usepackage{cite}
\usepackage[colorlinks,linkcolor=blue,anchorcolor=blue,citecolor=blue]{hyperref}

\newtheorem{theorem}{Theorem}[section]
\newtheorem{proposition}[theorem]{Proposition}
\newtheorem{lemma}[theorem]{Lemma}

\newtheorem*{Theorem}{Theorem}
\theoremstyle{definition}
\newtheorem{example}[theorem]{Example}
\theoremstyle{remark}
\newtheorem{remark}[theorem]{Remark}

\numberwithin{equation}{section}

\newcommand{\R}{\mathbb{R}}

\newcommand{\N}{\mathbb{N}}
\newcommand{\PP}{\mathbb{P}}
\newcommand{\E}{\mathbb{E}}
\newcommand{\f}{\infty}

\newcommand{\D}{\;\mathrm{d}}
\newcommand{\x}{\boldsymbol{x}}

\newcommand{\cale}{\mathcal{E}}
\newcommand{\calh}{\mathcal{H}}

\newcommand{\cals}{\mathcal{S}}

\title[Eigen-Falconer theorem in $\mathbb{R}^d$]{On the Eigen-Falconer theorem in $\mathbb{R}^d$}

\author[W. Li]{Wenxia Li}
\address[Wenxia Li]{School of Mathematical Sciences, Key Laboratory of MEA (Ministry of Education) $\&$ Shanghai Key Laboratory of PMMP, East China Normal University, Shanghai 200241, People's Republic of China}
\email{wxli@math.ecnu.edu.cn}

\author[Z. Wang]{Zhiqiang Wang}
\address[Zhiqiang Wang]{College of Mathematics and Statistics, Center of Mathematics, Chongqing University, Chongqing 401331, People's Republic of China \& Department of Mathematics, University of British Columbia, Vancouver, British Columbia, V6T 1Z2, Canada}
\email{zhiqiangwzy@163.com,~zqwangmath@cqu.edu.cn}

\author[J. Xu]{Jiayi Xu}
\address[Jiayi Xu]{School of Mathematical Sciences, Key Laboratory of MEA (Ministry of Education) $\&$ Shanghai Key Laboratory of PMMP, East China Normal University, Shanghai 200241, People's Republic of China}
\email{dkxujy@163.com}


\begin{document}

\begin{abstract}
In this paper, we study the analogous Erd\H{o}s similarity conjecture in higher dimensions and generalize the Eigen-Falconer theorem.
We show that if $A=\{\x_n\}_{n=1}^\f \subseteq \mathbb{R}^d$ is a sequence of non-zero vectors satisfying
\[ \lim_{n \to \f} \|\x_n\| =0 \quad \text{and} \quad \lim_{n \to \f} \frac{\|\x_{n+1}\|}{\|\x_n\|} = 1, \]
then there exists a measurable set $E \subseteq \mathbb{\R}^d$ with positive Lebesgue measure such that $E$ contains no affine copies of $A$.
\end{abstract}


\maketitle

\section{Introduction}

For a finite set $A \subseteq \R$, by the Lebesgue density theorem, any measurable subset $E \subset \R$ with positive Lebesgue measure contains a similar copy of $A$ (see \cite[Proposition 2.3]{Jung-Lai-Mooroogen-2025}). Here, a similar copy of $A$ is $\lambda A + t$ where $\lambda \ne 0$ and $t \in \R$.
In 1974, P. Erd\H{o}s suggested the following question \cite{Erdos-1974}, which now is known as the Erd\H{o}s similarity conjecture.
\begin{quote}
  \emph{Let $A \subseteq \R$ be an infinite set. Prove that there is a measurable subset of $\R$ with positive Lebesgue measure which does not contain a similar copy of $A$.}
\end{quote}
Although there is some important progress on the Erd\H{o}s similarity conjecture, it remains open even for any geometric sequence $A=\{ r^n \}_{n=1}^\f$ where $0< r < 1$.
We refer the reader to \cite{Svetic-2000,Jung-Lai-Mooroogen-2025} for an overview and some recent advancements of the Erd\H{o}s similarity conjecture.

Let $\mu$ denote the Lebesgue measure on $\R$. A set $A \subseteq \R$ is called an \emph{Erd\H{o}s set} if there exists a measurable subset $E \subset \R$ with $\mu(E)>0$ such that $\lambda A + t \not\subseteq E$ for any $\lambda \ne 0$ and any $t \in \R$.
In other words, the Erd\H{o}s similarity conjecture states that every infinite set in $\R$ is an Erd\H{o}s set.
Observe that any unbounded set is an Erd\H{o}s set.
Thus, if one could show that all strictly decreasing sequences converging to $0$ are Erd\H{o}s sets, then the Erd\H{o}s similarity conjecture would be fully settled.
The first significant progress was made independently by Eigen \cite{Eigen-1985} and Falconer \cite{Falconer-1984}.

\begin{Theorem}[Eigen-Falconer]
  Let $A=\{a_n\}_{n=1}^\f \subseteq \R$ be a strictly decreasing sequence converging to $0$. If
  \begin{equation}\label{eq:Eigen-Falconer-condition}
    \lim_{n \to \f} \frac{a_{n+1}}{a_n} = 1,
  \end{equation}
  then the set $A$ is an Erd\H{o}s set.
\end{Theorem}

Several subsequent papers attempted to weaken the condition (\ref{eq:Eigen-Falconer-condition}) \cite{Jasinski-1995,Kolountzakis-1997,Humke-Laczkovich-1998}.
Recently, Feng, Lai and Xiong \cite{Feng-Lai-Xiong-2024} showed that for a strictly decreasing sequence $A=\{a_n\}_{n=1}^\f$ converging to $0$, if (\ref{eq:Eigen-Falconer-condition}) holds, then there exists a compact set $E \subseteq \R$ with $\mu(E)>0$ such that $f(A) \not\subseteq \R$ for any bi-Lipschitz map $f:\R \to \R$.
They also proved that if \[\limsup_{n \to \f} \frac{a_{n+1}}{a_n} < 1,\] then for any measurable set $E\subseteq \R$ with $\mu(E)>0$, there exists a bi-Lipschitz map $f:\R \to \R$ such that $f(A) \subseteq E$.
Bourgain \cite{Bourgain-1987} proved that if $A_1,A_2,A_3 \subset \R$ are infinite sets, then the set $A_1 + A_2 + A_3 =\big\{ a_1 + a_2 + a_3: a_1 \in A_1, a_2 \in A_2, a_3 \in A_3 \big\}$ is an Erd\H{o}s set.
Using a probabilistic construction, Kolountzakis \cite{Kolountzakis-1997} can show that the sumset $\big\{ 2^{-n^\alpha} + 2^{-m^\alpha}: n, m \in \N \big\}$ is an Erd\H{o}s set for any $0< \alpha < 2$.
In the same paper, Kolountzakis proved that for any infinite set $A \subseteq \R$, there exists a measurable subset $E \subseteq \R$ with $\mu(E)>0$ such that the set $\big\{ (\lambda,t): \lambda A +t \subseteq E \big\}$ has two-dimensional Lebesgue measure zero, which can be viewed as the almost everywhere answer to the Erd\H{o}s similarity conjecture.

Analogous questions can be considered in higher dimensions.
We also use $\mu$ to denote the $d$-dimensional Lebesgue measure on $\R^d$, and let $\mathrm{GL}_d(\R)$ be the set of all $d\times d$ invertible matrices, which is identified with the set of all bijective linear transformations on $\R^d$.
For $T \in \mathrm{GL}_d(\R)$ and $\x \in \R^d$, the set $TA + \x:= \big\{ T \boldsymbol{a} + \x: \; \boldsymbol{a} \in A \big\}$ is called an \emph{affine copy} of $A$.
We present a generalized formulation of the Erd\H{o}s similarity conjecture in $\R^d$.

\noindent
\textbf{Generalized Erd\H{o}s Similarity Conjecture in $\R^d$}:
\emph{For an infinite set $A \subseteq \R^d$, there exists a measurable subset $E \subseteq \R^d$ with $\mu(E)>0$ such that $E$ contains no affine copies of $A$.}

It is worth pointing out that the generalized Erd\H{o}s similarity conjecture in $\R^d$ for some $d>1$ implies the original Erd\H{o}s similarity conjecture.
Suppose that the generalized Erd\H{o}s similarity conjecture in $\R^d$ holds for some $d>1$.
Let $A \subseteq \R$ be an infinite set. Define $\widetilde{A}=\{(a,0,\ldots,0): a \in A\} \subseteq \R^d$.
By assumption, there exists a measurable subset $E \subseteq \R^d$ with $\mu(E)>0$ such that $E$ contains no affine copies of $\widetilde{A}$.
For $\boldsymbol{y} \in \R^{d-1}$, define $E^{\boldsymbol{y}}=\{ x\in \R: (x,\boldsymbol{y}) \in E \}$.
By Fubini's theorem, we have \[ \mu(E) = \int_{\R^{d-1}} \mu(E^{\boldsymbol{y}}) \D \boldsymbol{y} >0, \]
where $\mu(E^{\boldsymbol{y}})$ denotes the $1$-dimensional Lebesgue measure of $E^{\boldsymbol{y}}$.
There must be $\boldsymbol{y}_0 \in \R^{d-1}$ such that $\mu(E^{\boldsymbol{y}_0})>0$.
Note that $\lambda A + t \subseteq E^{\boldsymbol{y}_0}$ if and only if $T \widetilde{A} + \x \subseteq E$ for $T=\mathrm{diag}(\lambda,\ldots,\lambda)$ and $\x=(t,\boldsymbol{y}_0)$.
Thus, we conclude that $E^{\boldsymbol{y}_0}$ does not contain a similar copy of $A$.
This means that the original Erd\H{o}s similarity conjecture holds.

The main purpose of this paper is to generalize the Eigen-Falconer theorem to higher dimensions.
Let $\|\boldsymbol{x}\|$ denote the usual Euclidean norm of a vector $\x \in \R^d$.

\begin{theorem}\label{thm:general-Eigen-Falconer}
  Let $A=\{\x_n\}_{n=1}^\f \subseteq \R^d$ be a sequence of non-zero vectors.
  If \[ \lim_{n \to \f} \|\x_n\| =0 \quad \text{and} \quad \lim_{n \to \f} \frac{\|\x_{n+1}\|}{\|\x_n\|} = 1, \]
  then there exists a closed set $E \subseteq [0,1]^d$ with $\mu(E)>0$ such that $E$ contains no affine copies of $A$.
\end{theorem}
\begin{remark}
  The Eigen-Falconer theorem is a corollary of Theorem \ref{thm:general-Eigen-Falconer}, and we will establish a slightly more general result (Theorem \ref{thm:general-Kolountzakis}).
\end{remark}

We illustrate Theorem \ref{thm:general-Eigen-Falconer} in $\R^2$ by an example.

\begin{example}
  Let $\{a_n\}_{n=1}^\f \subseteq \R$ be a positive sequence converging to $0$ and suppose that
  \[ \lim_{n \to \f} \frac{a_{n+1}}{a_n} = 1. \]
  For $n\in \N$, choose arbitrarily an point $(x_n,y_n) \in \R^2$ such that $x_n^2 + y_n^2 = a_n^2$.
  Let $A=\big\{ (x_n,y_n) \big\}_{n =1}^\f$.
  Then by Theorem \ref{thm:general-Eigen-Falconer}, there exists a closed set $E \subseteq \R^2$ with $\mu(E)>0$ such that $E$ contains no affine copies of $A$.
\end{example}

The rest of this paper is organized as follows.
In Section \ref{sec:general-Kolountzakis}, we state a slightly more general result (Theorem \ref{thm:general-Kolountzakis}), and prove Theorems \ref{thm:general-Eigen-Falconer} and \ref{thm:general-Kolountzakis} by assuming a key proposition (Proposition \ref{prop:main}).
The proof of Proposition \ref{prop:main} will be given in Section \ref{sec:proof-prop}.

\section{A generalization of Kolountzakis's result}\label{sec:general-Kolountzakis}

Let $\#A$ denote the cardinality of a set $A$.
For a finite set $A \subseteq \R^d$ with $\# A \geqslant 2$, define
\[ \delta(A) := \frac{\min\{\|\boldsymbol{x} -\boldsymbol{y}\|: \boldsymbol{x} \ne \boldsymbol{y} \in A \}}{\max\{\|\boldsymbol{z}\|: \boldsymbol{z} \in A\} }.  \]
Then we have $\delta(A) \leqslant 2$.
If $\{A_n\}_{n=1}^\f$ is a sequence of finite subsets of $\R^d$ with $\# A_n \to +\f$ as $n \to \f$, then \[ \lim_{n \to \f} \delta(A_n) =0. \]
The following theorem is a generalization of Kolountzakis's result \cite[Theorem 3]{Kolountzakis-1997} in higher dimensions.

\begin{theorem}\label{thm:general-Kolountzakis}
Let $A\subseteq\R^d$ be a bounded infinite set.
Suppose that there exists a sequence $\{A_n\}_{n=1}^\f$ of finite subsets of $A$ such that
\begin{equation}\label{eq:condition-A-n}
  \lim_{n \to \f} \# A_n = +\f \quad \text{and} \quad \lim_{n \to \f} \frac{-\log \delta(A_n)}{\# A_n} = 0.
\end{equation}
Then there exists a closed set $E\subseteq[0,1]^d$ with $\mu(E)>0$ such that $E$ contains no affine copies of $A$.
\end{theorem}
\begin{remark}
  {\rm(a)} If $A=\{\x_n\}_{n=1}^\f \subseteq \R^d$ is a sequence of non-zero vectors satisfying (\ref{eq:condition-A-n}), and $\{\|\x_n\|\}_{n=1}^\f$ is strictly decreasing, then we have
  \begin{equation}\label{eq:limsup-1}
    \limsup_{n \to \f} \frac{\|\x_{n+1}\|}{\| \x_n\|} = 1.
  \end{equation}
  This can be derived by a contradiction argument.
  Suppose that (\ref{eq:limsup-1}) does not hold. Then there exists $0< \rho < 1$ such that \[ \frac{\|\x_{n+1}\|}{\| \x_n\|} \leqslant \rho \quad \forall n \in \N. \]
  It follows that $\|\x_{n+k}\| \leqslant \rho^k \|\x_n\|$ for $n \in \N$ and $k \in \N$.
  Let $F\subseteq A$ be a finite subset. Write $F=\big\{ \x_{n_1}, \x_{n_2}, \ldots, \x_{n_k} \big\}$, where $n_1 < n_2 < \cdots < n_k$.
  For $1\leqslant i < j \leqslant n_k$, we have $\|\x_{n_i} - \x_{n_j}\| \leqslant \|\x_{n_i}\| +\| \x_{n_j}\| \leqslant \big( \rho^{n_i-n_1} + \rho^{n_j - n_1}\big) \|\x_{n_1}\| \leqslant (\rho^{i-1} + \rho^{j-1}) \|\x_{n_1}\|$.
  Thus, we obtain that $\delta(F) \leqslant \rho^{k-2} + \rho^{k-1} \leqslant 2 \rho^{k-2}$.
  It follows that \[ \frac{-\log \delta(F)}{\#F} \geqslant - \log \rho + \frac{2\log \rho - \log 2}{\#F}, \]
  which contradicts (\ref{eq:condition-A-n}).
  Thus, we obtain (\ref{eq:limsup-1}).

  {\rm(b)} We construct an example in $\R$ that satisfies Theorem \ref{thm:general-Kolountzakis} but not Theorem \ref{thm:general-Eigen-Falconer}, see Example \ref{example-2} for an example in $\R^2$.
  Choose two sequences $\{r_n\}_{n=1}^\f$ and $\{\rho_n\}_{n=1}^\f$ in $(0,1)$ such that $r_n \searrow 0$, $\rho_n \nearrow 1$, and
  \begin{equation}\label{eq:condition-rho-n}
    \lim_{n \to \f} \frac{\log (1-\rho_n)}{n} =0.
  \end{equation}
  For $n \in \N$, let $A_n = \big\{ r_1r_2 \cdots r_n \rho_1 \rho_2^2 \cdots \rho_{n-1}^{n-1} \rho_n^k \big\}_{k=0}^{n}$. Note that $\delta(A_n) = \rho_n^{n-1} - \rho_n^n$.
  By (\ref{eq:condition-rho-n}), we have \[ \lim_{n \to \f} \frac{-\log \delta(A_n)}{\# A_n} =0. \]
  Thus, the set $A = \bigcup_{n=1}^\f A_n$ satisfies (\ref{eq:condition-A-n}).
  Note that $r_n \to 0$. For any strictly decreasing sequence $\{a_n\}_{n=1}^\f \subseteq A$, we have \[ \liminf_{n\to\f}\frac{a_{n+1}}{a_n} =0. \]
\end{remark}

Assuming Theorem \ref{thm:general-Kolountzakis}, we can prove Theorem \ref{thm:general-Eigen-Falconer} now.
The following argument is similar with that in \cite[Subsection 4.3]{Kolountzakis-1997}.

\begin{proof}[Proof of Theorem \ref{thm:general-Eigen-Falconer}]
  Fix $n \in \N$, and let $\rho_n = 1-e^{-\sqrt{n}} $.
  Since
  \begin{equation*}
    \lim_{k\to\infty}\frac{\|\boldsymbol{x}_{k+1}\|}{\|\boldsymbol{x}_k\|}=1>\rho_n,
  \end{equation*}
  we can find $k_0 \in \N$ such that
  \begin{equation*}
    \frac{\|\boldsymbol{x}_{k+1}\|}{\|\boldsymbol{x}_k\|}>\rho_n \quad \forall k\geqslant k_0.
  \end{equation*}
  Choose $m \in \N$ such that $\rho_n^{m}\leqslant\|\boldsymbol{x}_{k_0}\|$.
  For each $j\in \N$, the interval $[\rho_n^{m+j}, \rho_n^{m + j -1})$ contains at least one point in $\{\|\boldsymbol{x}_k\|\}_{k=1}^{\infty}$. So we can choose a vector $\boldsymbol{a}_j$ from $A$ such that $\|\boldsymbol{a}_j\|\in[\rho_n^{m+j},\rho_n^{m+j-1})$.
  Let
  \begin{equation*}
    A_n=\{\boldsymbol{a}_1, \boldsymbol{a}_3, \ldots, \boldsymbol{a}_{2n-1}, \boldsymbol{a}_{2n+1}\}.
  \end{equation*}
  Then we have $\# A_n = n+1$, and
  \begin{align*}
    \delta(A_n) & = \frac{\min\big\{\|\boldsymbol{a}_{2i-1}-\boldsymbol{a}_{2j-1}\|:\ 1\leqslant i<j\leqslant n+1\big\}}{\max\big\{\|\boldsymbol{a}_{2\ell-1}\|:\ 1\leqslant \ell \leqslant n+1\big\}} \\
    & \geqslant \frac{\min\big\{\rho_n^{m+2i-1}-\rho_n^{m+2j-2}:\ 1\leqslant i<j\leqslant n+1\big\}}{ \rho_n^m} \\
    & = \frac{\rho_n^{m+2n-1} - \rho_n^{m+2n}}{\rho_n^m}= \rho_n^{2n-1}(1-\rho_n).
  \end{align*}

  It follows that
  \begin{equation*}
    \frac{-\log\delta(A_n)}{\# A_n}\leqslant -\frac{\log(1-\rho_n)}{n+1}-\frac{2n-1}{n+1}\log\rho_n\longrightarrow0 \;\;\;\; \text{as}\; n \to \f.
  \end{equation*}
  Thus, we obtain \[ \lim_{n \to \f} \frac{-\log\delta(A_n)}{\# A_n} = 0. \]
  By Theorem \ref{thm:general-Kolountzakis}, there exists a closed subset $E \subseteq [0,1]^d$ with $\mu(E) > 0$ such that $E$ contains no affine copies of $A$.
\end{proof}

For $T \in \mathrm{GL}_d(\R)$, define
\[ \|T\|^{*} :=\max_{\|\x\|=1} \|T\boldsymbol{x}\| \quad \text{and} \quad \|T\|_{*} :=\min_{\|\x\|=1} \|T\boldsymbol{x}\|. \]
Thus, we have \[ \|T\|_{*}\|\x\| \leqslant \| T \x \| \leqslant \|T\|^{*} \|\x\| \quad \forall \x \in \R^d. \]
For $0<\alpha<\beta$, define
\[ \cals_{\alpha}^{\beta} :=\big\{ T\in\mathrm{GL}_d(\R):\; \alpha < \|T\|_{*} \leqslant \|T\|^{*} < \beta \big\}. \]
The proof of Theorem \ref{thm:general-Kolountzakis} relies on the following key proposition.

\begin{proposition}\label{prop:main}
  Let $A\subseteq\R^d$ be a bounded infinite set, and suppose that there exists a sequence $\{A_n\}_{n=1}^\f$ of finite subsets of $A \setminus\{\boldsymbol{0}\}$ satisfying (\ref{eq:condition-A-n}).
  Then for any $0<\alpha<1$, there exists a sequence $\{E_n\}_{n=1}^{\infty}$ of open subsets of $[0, 1]^d$ such that
  \begin{equation*}
    \lim_{n \to \f} \mu(E_n) = 1 \quad \text{and} \quad \lim_{n \to \f} \mu^{*}(V_n) = 0,
  \end{equation*}
  where $V_n:=\big\{ \boldsymbol{x}\in [0,1]^d:\; \text{there exists}\; T\in\cals_{\alpha}^{1/\alpha}\; \text{such that}\; TA + \x \subseteq E_n \big\}$, and $\mu^{*}(V_n)$ is the Lebesgue outer measure of $V_n$.
\end{proposition}

The detailed proof of Proposition \ref{prop:main} will be given in Section \ref{sec:proof-prop}.
Now we prove Theorem \ref{thm:general-Kolountzakis} by using Proposition \ref{prop:main}.

\begin{proof}[Proof of Theorem \ref{thm:general-Kolountzakis}]
Note that $\delta\big( A_n \setminus \{ \boldsymbol{0} \} \big) \geqslant \delta(A_n)$.
It follows from (\ref{eq:condition-A-n}) that \[ \lim_{n \to \f} \frac{-\log \delta\big( A_n \setminus \{ \boldsymbol{0} \} \big)}{\# \big( A_n \setminus \{ \boldsymbol{0} \} \big)} =0. \]
Thus, we can always assume that $\boldsymbol{0} \not\in A_n$ for all $n \in \N$.

We first assume that $\boldsymbol{0} \in A$.
Fix $0< \alpha < 1$, and let $\{E_n\}_{n=1}^\f$ and $\{V_n\}_{n=1}^\f$ be defined in Proposition \ref{prop:main}.
For $n\in\N$, by the inner regularity of the Lebesgue measure, there exists a closed set $F_n\subseteq E_n$ with $\mu(F_n)>\mu(E_n)-\frac{1}{n}$, and by the definition of the Lebesgue outer measure, we can find an open set $U_n\supseteq V_n$ with $\mu(U_n)<\mu^{*}(V_n)+\frac{1}{n}$.
Let $\widetilde{E}_n=F_n\setminus U_n$.
Then, $\widetilde{E}_n$ is a closed subset of $E_n$ and
\begin{equation}\label{eq:measure-tilde-E}
  \mu\big( \widetilde{E}_n \big) > \mu(E_n) - \mu^{*}(V_n) - \frac{2}{n}.
\end{equation}
For $T\in\cals_{\alpha}^{1/\alpha}$ and $\boldsymbol{x}\in \R^d$, if $\x \in V_n$ or $\x \not\in [0,1]^d$, then $\x \not\in \widetilde{E}_n$ and it follows that $TA + \x\not\subseteq\widetilde{E}_n$ because $\boldsymbol{0}\in A$; if $\x \in [0,1]^d \setminus V_n$, then we have $TA +\x \not \subseteq E_n$, which implies $TA+\x\not\subseteq\widetilde{E}_n$.
Thus, we conclude that $TA+\x\not\subseteq\widetilde{E}_n$ for any $T\in\cals_{\alpha}^{1/\alpha}$ and any $\boldsymbol{x}\in \R^d$.
By Proposition \ref{prop:main} and (\ref{eq:measure-tilde-E}), we have \[ \lim_{n \to \f} \mu\big( \widetilde{E}_n \big)=1. \]
By choosing a large enough integer, we can obtain a closed subset $E_\alpha$ of $[0,1]^d$ with $\mu(E_\alpha)>1-\alpha$ such that $TA+\x\not\subseteq E_\alpha$ for any $T\in\cals_{\alpha}^{1/\alpha}$ and any $\boldsymbol{x}\in \R^d$.

For $k \in \N$, let $\alpha_k = 1/4^k$.
By the previous argument, we obtain a sequence $\{E_{\alpha_k}\}_{k=1}^{\infty}$ of closed subsets of $[0,1]^d$ satisfying $\mu\big(E_{\alpha_k}\big)>1-4^{-k}$ and
\[ TA +\x \not\subseteq E_{\alpha_k}\;\; \forall T\in\cals_{\alpha_k}^{1/\alpha_k}\; \forall\boldsymbol{x}\in\R^d. \]
We claim that the intersection
\begin{equation*}
  E=\bigcap_{k=1}^{\infty}E_{\alpha_k}
\end{equation*}
is the desired set. To see this, we first have that $E$ is a closed subset of $[0, 1]^d$, and
\begin{align*}
  \mu(E) = 1 - \mu\big([0, 1]^d\setminus E\big) \geqslant 1- \sum_{k=1}^{\infty}\mu\big([0, 1]^d\setminus E_{\alpha_k}\big) \geqslant 1 - \sum_{k=1}^{\f} \frac{1}{4^k} = \frac{2}{3}.
\end{align*}
For $T\in\mathrm{GL}_d(\R)$ and $\boldsymbol{x}\in \R^d$, since $0< \|T\|_{*} \leqslant \|T\|^{*} < +\f$, we can find a sufficiently large integer $k \in \N$ such that $T \in \cals_{\alpha_k}^{1/\alpha_k}$.
Note that $TA+\x \not \subseteq E_{\alpha_k}$ and $E \subseteq E_{\alpha_k}$.
Thus, we have $TA+\x \not \subseteq E$.
That is, the set $E$ contains no affine copies of $A$.

Next, we assume that $\boldsymbol{0} \not\in A$.
If $\boldsymbol{0}$ is an accumulation point of $A$, then let $\widetilde{A} = A\cup\{ \boldsymbol{0} \}$.
By the previous argument, there exists a closed subset $E \subseteq [0,1]^d$ with $\mu(E) > 0$ such that $E$ contains no affine copies of $\widetilde{A}$.
Note that $E$ is closed, and $\mathbf{0}$ is an accumulation point of $A$.
Thus, the set $E$ contains no affine copies of $A$.

If $\boldsymbol{0}$ is not an accumulation point of $A$, then there exists $C>0$ such that $\|\boldsymbol{a}\|\geqslant C$ for all $\boldsymbol{a} \in A$.
Choose $\boldsymbol{a}_0 \in A$ and let $\widetilde{A}= A-\boldsymbol{a}_0$ and $\widetilde{A}_n = A_n - \boldsymbol{a}_0$.
Clearly, we have $\#\widetilde{A}_n = \#A_n$.
Note that $\max\{\|\boldsymbol{z}\|: \boldsymbol{z} \in A_n \} \geqslant C$.
We have \[ \max\{\|\boldsymbol{z}\|: \boldsymbol{z} \in \widetilde{A}_n \} \leqslant \max\{\|\boldsymbol{z}\|: \boldsymbol{z} \in A_n \} + \|\boldsymbol{a}_0\| \leqslant \bigg( 1 + \frac{\|\boldsymbol{a}_0\|}{C} \bigg) \max\{\|\boldsymbol{z}\|: \boldsymbol{z} \in A_n \}. \]
Let $\widetilde{C} = C/(C+\|\boldsymbol{a}_0\|)$. Then we obtain $\delta(\widetilde{A}_n) \geqslant \widetilde{C} \delta(A_n)$.
It follows that $-\log \delta(\widetilde{A}_n) \leqslant - \log \delta(A_n) - \log \widetilde{C}$.
Thus, by (\ref{eq:condition-A-n}), we have \[ \lim_{n \to \f} \frac{-\log \delta(\widetilde{A}_n)}{\# \widetilde{A}_n }=0. \]
By the previous argument, there exists a closed subset $E \subseteq [0,1]^d$ with $\mu(E) > 0$ such that $E$ contains no affine copies of $\widetilde{A}$.
Clearly, the set $E$ contains no affine copies of $A$.
We complete the proof.
\end{proof}

Finally, we give an example in $\R^2$.
\begin{example}\label{example-2}
  Let $\{ a_n \}_{n=1}^\f$ be an arbitrary positive sequence.
  For $n \in \N$, let $A_n$ be the vertices of an inscribed equilateral $(n+1)$-polygon of the circle $\big\{ (x,y) \in \R^2: x^2+y^2 = a_n^2 \big\}$.
  It is easy to calculate that $\delta(A_n) = 2 \sin \frac{\pi}{n+1}$.
  So we have \[ \lim_{n \to \f} \frac{-\log \delta(A_n)}{\# A_n} = \lim_{n \to \f} -\frac{1}{n+1} \log \bigg( 2\sin \frac{\pi}{n+1} \bigg) = 0.  \]
  Let $A= \bigcup_{n=1}^\f A_n$.
  Then there exists a measurable set $E \subseteq \R^d$ with $\mu(E)>0$ such that $E$ contains no affine copies of $A$.
  If $A$ is unbounded, then the conclusion is clear; if $A$ is bound, then the conclusion follows from Theorem \ref{thm:general-Kolountzakis} directly.
  Take $a_n = 1/2^n$ and perturb each point in $A_n$ slightly so that all vectors in $A$ have distinct norms. This yields an example $A \subseteq \R^2$ that satisfies Theorem \ref{thm:general-Kolountzakis} but not Theorem \ref{thm:general-Eigen-Falconer}. 
\end{example}

\section{Proof of Proposition \ref{prop:main}}\label{sec:proof-prop}

In this section, we will prove Proposition \ref{prop:main}.
We always assume that $A\subseteq\R^d$ is a bounded infinite set, and there exists a sequence $\{A_n\}_{n=1}^\f$ of finite subsets of $A \setminus \{\boldsymbol{0}\}$ satisfying (\ref{eq:condition-A-n}).
Write $k_n = \#A_n$ and $\delta_n = \delta(A_n)$ for $n \in \N$.
Then we have
\begin{equation}\label{eq:condition-A-n-2}
  \lim_{n \to \f} k_n = +\f \quad \text{and} \quad \lim_{n \to \f} \frac{-\log \delta_n}{k_n} = 0.
\end{equation}
We also fix $0<\alpha < 1$ in the following.

The proof of Proposition \ref{prop:main} is based on a probability construction developed by Kolountzakis in \cite{Kolountzakis-1997}.
The main difficulty we met in higher dimensions is to reduce the set $\cals_{\alpha}^{1/\alpha}$ to a finite set.
To this end, we partition the set of all $d\times d$ real matrices by some hyperplanes and take a representative element from each connected component (see Lemma \ref{lemma:finiteness}).

A \emph{hyperplane} of $\R^d$ is a $(d-1)$-dimensional affine subspace of $\R^d$, which can be defined by \[ H=\left\{\boldsymbol{x}\in\R^d:\ \boldsymbol{v}\cdot\boldsymbol{x}+b=0\right\}, \]
where $\boldsymbol{v}\in \R^d$ with $\|\boldsymbol{v}\|=1$ and $b\in\R$.
We say that \emph{$\boldsymbol{y}, \boldsymbol{z} \in \R^d$ lie in the same side of $H$} if $\boldsymbol{v}\cdot\boldsymbol{y}+b$ and $\boldsymbol{v}\cdot\boldsymbol{z}+b$ have the same sign.
For $\ell\in\{1, 2, \ldots, d\}$ and $b \in \R$, define
\begin{equation}\label{eq:hyperplane-H-ell-b}
  H_{\ell,b} :=\R^{\ell-1}\times\{b\}\times\R^{d-\ell} = \big\{ \boldsymbol{x}\in\R^d:\ \boldsymbol{e}_{\ell}\cdot\boldsymbol{x} - b=0 \big\},
\end{equation}
where $\boldsymbol{e}_{\ell} =(0, \ldots, 0, 1, 0, \ldots, 0)$ is the $\ell$-th standard orthonormal basis with $1$ at the $\ell$-th position.

For $n\in\N$, define $M_{n}:=\max\{\|\boldsymbol{x}\|: \boldsymbol{x}\in A_n\}$ and \[ L_{n}:=\left\lceil\frac{d}{\alpha M_n\delta_n}\right\rceil, \] where $\lceil x\rceil$ denotes the smallest integer larger than or equal to $x$.
We divide the unite hypercube $[0,1]^d$ into open sub-hypercubes with side length $1/L_n$.
For $j_1,j_2, \ldots, j_d\in\{0,1,\ldots, L_{n}-1\}$, define
\begin{equation*}
  I_{j_1,j_2, \ldots, j_d}(n):=\left(\frac{j_1}{L_{n}}, \frac{j_1+1}{L_{n}}\right)\times\left(\frac{j_2}{L_{n}}, \frac{j_2+1}{L_{n}}\right)
  \times\cdots\times\left(\frac{j_d}{L_{n}}, \frac{j_d+1}{L_{n}}\right).
\end{equation*}
Let $\Omega_n$ be the set of $0$-$1$ sequences with length $(L_n)^d$.
Each element in $\Omega_n$ can be viewed as a map from open hypercubes $\big\{ I_{j_1,j_2, \ldots, j_d}(n):\; j_1,j_2, \ldots, j_d\in\{0,1,\ldots, L_{n}-1\} \big\}$ to $\{0,1\}$.
So, the element in $\Omega_n$ will be denoted by $\boldsymbol{\omega}=(\omega_{j_1,j_2, \ldots, j_d})_{0 \leqslant j_1,j_2, \ldots, j_d\leqslant L_n-1}$.

For $\boldsymbol{\omega}=(\omega_{j_1,j_2, \ldots, j_d})_{0 \leqslant j_1,j_2, \ldots, j_d\leqslant L_n-1}\in\Omega_n$, define
\[ \cale_n(\boldsymbol{\omega}):=\bigcup_{\substack{0\leqslant j_1,j_2, \ldots, j_d\leqslant L_n-1\\ \omega_{j_1,j_2, \ldots, j_d}=1}} I_{j_1,j_2, \ldots, j_d}(n). \]
Then, $\cale_n$ is a map from $\Omega_n$ to open subsets of $[0, 1]^d$.
For $\boldsymbol{x}\in[0, 1]^d$, define
\begin{equation}\label{eq:def-W-x-n}
  W_{\boldsymbol{x},n}:= \big\{\boldsymbol{\omega}\in\Omega_n:\ \text{there exists}\ T\in\cals_{\alpha}^{1/\alpha}\  \text{such that}\ TA_n + \boldsymbol{x} \subseteq\cale_n(\boldsymbol{\omega})\big\}.
\end{equation}
We first reduce the set $\cals_{\alpha}^{1/\alpha}$ in the definition of $W_{\boldsymbol{x},n}$ in (\ref{eq:def-W-x-n}) to a finite set.

\begin{lemma}\label{lemma:finiteness}
  For $n\in\N$ and $\boldsymbol{x}\in[0, 1]^d$, there exists a finite subset $\cals_n(\boldsymbol{x})\subseteq\cals_{\alpha}^{1/\alpha}$ such that
  \begin{equation}\label{eq:reduce-W-x-n-finite}
    W_{\boldsymbol{x},n}=\big\{\boldsymbol{\omega}\in\Omega_n:\ \text{there exists}\ S\in\cals_n(\boldsymbol{x})\ \text{such that}\ SA_n + \boldsymbol{x} \subseteq\cale_n(\boldsymbol{\omega})\big\},
  \end{equation}
  and \[ \#\cals_n(\boldsymbol{x})\leqslant C(L_nM_nk_n)^{d^2}, \] where $C$ is a constant depending only on $\alpha$ and $d$.
\end{lemma}

To prove the lemma, we need the following lemma to estimate the number of connected components arising from partitioning $\R^d$ by its hyperplanes.
\begin{lemma}[Buck \cite{Buck-1943}]\label{lemma:partition-space}
  Let $\calh$ be the set of $n$ hyperplanes in $\R^d$. Then the number of connected regions of $\R^d \setminus \bigcup\calh$ is at most \[ \sum_{k=0}^{d}\binom{n}{k}, \]
  which has a trivial upper bound $(d+1)n^d$.
\end{lemma}

\begin{proof}[Proof of Lemma \ref{lemma:finiteness}]
  Fix $n\in\N$ and $\boldsymbol{x}\in[0, 1]^d$.
  Write $A_n = \big\{ \x_1, \x_2, \ldots, \x_{k_n} \big\}$.
  Let $\mathrm{M}_d(\R)$ be the set of all $d\times d$ real matrices, which can be identified with $\R^{d^2}$.
  Clearly, $\mathrm{GL}_d(\R) \subseteq \mathrm{M}_d(\R)$.
  Recall the definition of $H_{\ell,b}$ in (\ref{eq:hyperplane-H-ell-b}).
  For $\ell\in\{1, 2, \ldots, d\}$, $i\in\{1, 2, \ldots, k_n\}$, and $j\in\{0, 1, \ldots, L_n\}$, let
  \begin{align*}
    \widetilde{H}_{\ell, i,j} &:=\big\{T\in \mathrm{M}_d(\R):\ T\boldsymbol{x}_i + \boldsymbol{x}\in H_{\ell,j/L_n}\big\} \\
     & = \big\{ T\in \mathrm{M}_d(\R):\ \boldsymbol{e}_\ell \cdot (T\boldsymbol{x}_i + \boldsymbol{x}) - j/L_n=0 \big\}.
  \end{align*}
  Since $\boldsymbol{x}_i \ne \boldsymbol{0}$, it is easy to check that $\widetilde{H}_{\ell,i,j}$ is a $(d^2-1)$-dimensional affine subspace of $\R^{d^2}$, i.e., a hyperplane of $\R^{d^2}$.
  For $\ell\in\{1, 2, \ldots, d\}$ and $i\in\{1, 2, \ldots, k_n\}$, define
  \begin{equation}\label{eq:Lambda-ell-i}
    \Lambda_{\ell,i}:=\bigg\{ 0\leqslant j\leqslant L_n:\ \mathrm{dist}(\boldsymbol{x},H_{\ell,j/L_n})<\frac{\|\boldsymbol{x}_i\|}{\alpha}+\frac{1}{L_n}\bigg\}.
  \end{equation}
  Let \[ \calh= \Big\{ \widetilde{H}_{\ell, i,j} : \; 1\leqslant \ell \leqslant d,\; 1 \leqslant i \leqslant k_n,\; j \in \Lambda_{\ell,i} \Big\}. \]
  We partition $\R^{d^2}$ by hyperplanes in $\mathcal{H}$, and all connected regions of $\R^{d^2}\setminus \bigcup\calh$ are denoted by $R_1, R_2, \ldots, R_m$.
  For $1\leqslant k \leqslant m$, if $R_k \cap \cals_{\alpha}^{1/\alpha} \ne \varnothing$, then we choose an element from the set $R_k \cap \cals_{\alpha}^{1/\alpha}$. All these chosen elements make up the set $\cals_n(\boldsymbol{x})$.

  We first estimate the cardinality of $\cals_n(\boldsymbol{x})$.
  By Lemma \ref{lemma:partition-space}, we have
  \begin{equation}\label{eq:cardinality-S-x}
    \# \cals_n(\boldsymbol{x})\leqslant m \leqslant (d^2+1)(\#\calh)^{d^2}.
  \end{equation}
  By (\ref{eq:Lambda-ell-i}), we have
  \[ \# \Lambda_{\ell,i} \leqslant 2\bigg(\frac{\|\boldsymbol{x}_i\|L_n}{\alpha}+1 \bigg)+1 \leqslant \frac{2M_nL_n}{\alpha} +3. \]
  Note that $\delta_n \leqslant 2$. We have \[ L_n=\left\lceil\frac{d}{\alpha M_n \delta_n}\right\rceil\geqslant\frac{d}{\alpha M_n \delta_n} \geqslant \frac{1}{2M_n}, \]
  i.e., $2M_n L_n \geqslant 1$. So, we obtain $\# \Lambda_{\ell,i} \leqslant (8M_n L_n)/\alpha$.
  It follows that
  \[\#\calh \leqslant \sum_{\ell=1}^{d}\sum_{i=1}^{k_n}\#\Lambda_{\ell,i} \leqslant  \frac{8dM_nL_nk_n}{\alpha}, \]
  and hence by (\ref{eq:cardinality-S-x}), \[ \# \cals_n(\boldsymbol{x})\leqslant (d^2+1) \bigg( \frac{8dM_nL_nk_n}{\alpha} \bigg)^{d^2}. \]

  To prove (\ref{eq:reduce-W-x-n-finite}), it suffices to show that for $\boldsymbol{\omega} \in W_{\x,n}$ there exists $S\in\cals_n(\x)$ such that $SA_n + \boldsymbol{x} \subseteq \cale_n(\boldsymbol{\omega})$.
  For $1\leqslant\ell\leqslant d$, $1\leqslant i\leqslant k_n$, $j\in\Lambda_{i}^{\ell}$, and $ 1\leqslant k\leqslant m$, we have
  \begin{equation}\label{eq:H-ell-j-emptyset}
    \big\{ T\boldsymbol{x}_i +\x :\ T\in R_k \big\} \cap H_{\ell,j/L_n}= \emptyset.
  \end{equation}
  Note that the function $T \mapsto \boldsymbol{e}_\ell \cdot (T\boldsymbol{x}_i+ \boldsymbol{x}) - j/L_n$ is continuous on $\mathrm{M}_d(\R)$, and $R_k$ is connected in $\mathrm{M}_d(\R)$.
  So, the set $\big\{ \boldsymbol{e}_\ell \cdot (T\boldsymbol{x}_i+ \boldsymbol{x}) - j/L_n: T \in R_k \big\}$ is an interval in $\R$, which does not contain $0$ by (\ref{eq:H-ell-j-emptyset}).
  Thus, we conclude that all points in $\big\{ T\boldsymbol{x}_i +\x :\ T\in R_k \big\}$ lie in the same side of $H_{\ell,j/L_n}$.

  Next, we fix $\boldsymbol{\omega} \in W_{\x,n}$.
  There exists $T\in\cals_{\alpha}^{1/\alpha}$ such that $TA_n + \boldsymbol{x} \subseteq \cale_n(\boldsymbol{\omega})$.
  Then we can find $1 \leqslant k \leqslant m$ and $S\in\cals_n(\x)$ such that $T,S \in R_k$.
  For $1 \leqslant i \leqslant k_n$, there exist $j_1, j_2, \ldots, j_d \in\{0,1,\ldots, L_n-1\}$ such that
  \begin{equation}\label{eq:T-x-i-in-E-omega}
    T\boldsymbol{x}_i + \boldsymbol{x} \in I_{j_1, j_2, \ldots, j_d}(n) \subseteq \cale_n(\boldsymbol{\omega}).
  \end{equation}
  For $1\leqslant \ell \leqslant d$ and $j \in \{j_\ell, j_\ell+1\}$, we have $\mathrm{dist}\big( T\boldsymbol{x}_i + \boldsymbol{x}, H_{\ell,j/L_n} \big) < 1/L_n$. Note that $\|T\|^* < 1/\alpha$.
  It follows that \[ \mathrm{dist}\big( \boldsymbol{x}, H_{\ell,j/L_n} \big) < \|T \x_i\| + \frac{1}{L_n} < \frac{\|\x_i\|}{\alpha} + \frac{1}{L_n}. \]
  This implies that $ \{j_\ell, j_\ell+1\} \subset \Lambda_{\ell,i}$.
  Note that $T,S \in R_k$.
  Thus, the points $T\boldsymbol{x}_i +\x$ and $S\boldsymbol{x}_i +\x$ lie in the same side of $H_{\ell,j/L_n}$ for all $1\leqslant \ell \leqslant d$ and $j \in \{j_\ell, j_\ell+1\}$.
  By (\ref{eq:T-x-i-in-E-omega}), we conclude that \[ S\boldsymbol{x}_i + \boldsymbol{x} \in I_{j_1, j_2, \ldots, j_d}(n) \subseteq \cale_n(\boldsymbol{\omega}) \quad \forall 1 \leqslant i \leqslant k_n. \]
  That is, $SA_n + \x \subseteq \cale_n(\boldsymbol{\omega})$.
  The proof is completed.
\end{proof}

Next, we equip $\Omega_n$ with a probability measure to make it a probability space.
By (\ref{eq:condition-A-n-2}), we can find a sequence $\{p_n\}_{n=1}^{\infty}$ with $0< p_n < 1$ such that
\begin{equation}\label{eq:p-n-to-1}
  \lim_{n\to \f} p_n = 1,
\end{equation}
and
\begin{equation}\label{eq:sequence-p-n}
  \log p_n < \frac{d^2\log\delta_n}{k_n}-\frac{(d^2+1)\log k_n}{k_n}.
\end{equation}
For each $n\in\N$, we equip $\Omega_n$ with the Bernoulli product probability measure $\PP_n$ induced by probability vector $(1-p_n, p_n)$.

The set $\cale_n$ can be viewed as a random open set obtained by choosing independently sub-hypercubes $I_{j_1, \ldots, j_d}(n)$ for $j_1, \ldots, j_d \in \{0,1,\ldots,L_n-1\}$ with probability $p_n$.
The Lebesgue measure of $\cale_n$ is a random variable on $(\Omega_n, \PP_n)$, denoted by $\mu \circ \cale_n$.
For $n \in \N$ and $\boldsymbol{\omega} \in \Omega_n$, define
\begin{equation}\label{eq:def-V-n-omega}
  \mathcal{V}_n(\boldsymbol{\omega}):=\big\{ \boldsymbol{x}\in [0,1]^d:\ \text{there exists}\ T\in\cals_{\alpha}^{1/\alpha}\ \text{such that}\ TA_n + \boldsymbol{x} \subseteq\cale_n(\boldsymbol{\omega}) \big\}.
\end{equation}
Note that \[ \mathcal{V}_n(\boldsymbol{\omega}) = [0,1]^d \bigcap \bigg( \bigcup_{ T\in\cals_{\alpha}^{1/\alpha} } \bigcap_{\boldsymbol{a} \in A_n} \big( \cale_n(\boldsymbol{\omega}) - T \boldsymbol{a}\big) \bigg). \]
Since $\cale_n$ is a random open set, $\mathcal{V}_n$ is a random Borel subset of $[0,1]^d$.
The Lebesgue measure of $\mathcal{V}_n$ is also a random variable on $(\Omega_n, \PP_n)$, denoted by $\mu \circ \mathcal{V}_n$.
Let $\E$ denote the expectation of a random variable.

\begin{lemma}\label{lemma:limit-E-n-V-n}
  We have \[ \lim_{n \to \f} \E(\mu\circ\cale_n) =1 \quad \text{and} \quad \lim_{n \to \f} \E(\mu\circ\mathcal{V}_n) = 0. \]
\end{lemma}
\begin{proof}
  For $j_1, \ldots, j_d \in \{0,1,\ldots, L_n-1\}$, we have \[ \PP_n\big\{ \boldsymbol{\omega} \in \Omega_n:\ I_{j_1, \ldots, j_d}(n) \subseteq \cale_n(\boldsymbol{\omega}) \big\}=p_n. \]
  Note that the events $\big\{ \boldsymbol{\omega} \in \Omega_n:\ I_{j_1, \ldots, j_d}(n) \subseteq \cale_n(\boldsymbol{\omega})\big\}, j_1, \ldots, j_d \in \{0,1,\ldots, L_n-1\},$ are independent in $(\Omega_n,\PP_n)$.
  Thus, we have
  \[ \E(\mu\circ\cale_n) = \sum_{j_1, \ldots, j_d \in \{0,1,\ldots, L_n-1\}} \mu\big( I_{j_1, \ldots, j_d}(n) \big) \cdot \PP_n\big\{ \boldsymbol{\omega} \in \Omega_n:\ I_{j_1, \ldots, j_d}(n) \subseteq \cale_n(\boldsymbol{\omega})\big\} = p_n.  \]
  By (\ref{eq:p-n-to-1}), we have \[ \lim_{n \to \f} \E(\mu\circ\cale_n) =1.\]

  Let $\mathbbm{1}_{F}$ denote the indicator function of a set $F$.
  By (\ref{eq:def-W-x-n}) and (\ref{eq:def-V-n-omega}), one can check that
  \[ \mathbbm{1}_{W_{\x,n}}(\boldsymbol{\omega}) = \mathbbm{1}_{\mathcal{V}_n(\boldsymbol{\omega})} (\x)\quad \forall \x \in [0,1]^d\quad \forall \boldsymbol{\omega} \in \Omega_n. \]
  So, we have
  \begin{align}
    \E(\mu\circ\mathcal{V}_n) & = \int_{\boldsymbol{\omega} \in \Omega_n} \mu\big( \mathcal{V}_n( \boldsymbol{\omega}) \big) \D \PP_n(\boldsymbol{\omega}) \notag\\
    & = \int_{\boldsymbol{\omega} \in \Omega_n} \int_{\x \in [0,1]^d} \mathbbm{1}_{\mathcal{V}_n( \boldsymbol{\omega}) }(\x) \D \mu(\x) \D \PP_n(\boldsymbol{\omega}) \notag\\
    & = \int_{\boldsymbol{\omega} \in \Omega_n} \int_{\x \in [0,1]^d} \mathbbm{1}_{W_{\x,n}}(\boldsymbol{\omega}) \D \mu(\x) \D \PP_n(\boldsymbol{\omega}) \notag\\
    & =  \int_{\x \in [0,1]^d} \PP_n \big( W_{\x,n} \big) \D \mu(\x). \label{eq:E-mu-V-n}
  \end{align}
  Next we need to estimate $\PP_n\big(W_{\boldsymbol{x},n}\big)$ for $\boldsymbol{x}\in[0,1]^d$.

Fix $n\in\N$ and $\boldsymbol{x}\in[0, 1]^d$.
Write $A_n = \big\{ \x_1, \x_2, \ldots, \x_{k_n} \big\}$.
Let $\cals_n(\boldsymbol{x})$ be the finite subset of $\cals_{\alpha}^{1/\alpha}$ defined in Lemma \ref{lemma:finiteness}.
For $T\in\cals_n(\boldsymbol{x})$, noting that $\|T\|_* > \alpha$, we have
\begin{align*}
  & \;\min\{\|T\boldsymbol{x}_i-T\boldsymbol{x}_j\|:\ 1\leqslant i<j\leqslant k_n\} \\
   > &\; \alpha\min\{\|\boldsymbol{x}_i-\boldsymbol{x}_j\|:\ 1\leqslant i<j\leqslant k_n\} \\
  = & \; \alpha M_n \delta_n \geqslant\frac{d}{L_n}.
\end{align*}
This means that for $1\leqslant i<j\leqslant k_n$, the points $T\x_i + \x$ and $T\x_j + \x$ cannot lie in the same sub-hypercube with side length $1/L_n$.
This implies that the events $\big\{ \boldsymbol{\omega}\in\Omega_n:\ T\boldsymbol{x}_i + \boldsymbol{x} \in\cale_n(\boldsymbol{\omega}) \big\}, 1 \leqslant i \leqslant k_n,$ are independent in $(\Omega_n, \PP_n)$.
Thus, we obtain
\begin{align*}
  \PP_n \big\{ \boldsymbol{\omega} \in \Omega_n: TA_n +\x \subseteq\cale_n(\boldsymbol{\omega})\big\} & =\PP_n\bigg( \bigcap_{i=1}^{k_n}\big\{\boldsymbol{\omega}\in\Omega_n:\ T\boldsymbol{x}_i +\x \subseteq\cale_n(\boldsymbol{\omega}) \big\} \bigg) \\
  & =\prod_{i=1}^{k_n}\PP_n\big\{\boldsymbol{\omega}\in\Omega_n:\ T\boldsymbol{x}_i +\x \subseteq\cale_n(\boldsymbol{\omega})\big\} \\
  & \leqslant p_n^{k_n}.
\end{align*}
It follows from Lemma \ref{lemma:finiteness} that
\begin{align*}
  \PP_n(W_{\boldsymbol{x},n}) & =\PP_n \bigg( \bigcup_{T\in\cals_n(\boldsymbol{x})} \big\{ \boldsymbol{\omega}\in\Omega_n:\ TA_n+\x \subseteq \cale_n(\boldsymbol{\omega})\big\} \bigg) \\
  & \leqslant \sum_{T\in\cals_n(\boldsymbol{x})} \PP_n\big\{\boldsymbol{\omega}\in\Omega_n:\ TA_n +\x \subseteq\cale_n(\boldsymbol{\omega})\big\} \\
  & \leqslant p_n^{k_n} \cdot \#\cals_n(\boldsymbol{x}) \\
  & \leqslant C(L_n M_{n}k_{n})^{d^2}p_n^{k_n},
\end{align*}
where $C$ is the constant in Lemma \ref{lemma:finiteness}.
By (\ref{eq:E-mu-V-n}), we conclude that
\begin{equation}\label{eq:E-mu-V-n-upper-bound}
  \E(\mu\circ\mathcal{V}_n) \leqslant C(L_n M_{n}k_{n})^{d^2}p_n^{k_n}.
\end{equation}

Since the set $A$ is bounded, we can find $M>0$ such that $M_n \leqslant M$ for all $n \in \N$.
Note that $\delta_n \leqslant 2$.
We have
\begin{equation*}
  L_n\leqslant\frac{d}{\alpha M_n \delta_n}+1 \leqslant\frac{d}{\alpha M_n \delta_n} +\frac{2M}{M_n \delta_n} = \frac{\widetilde{C}}{M_n \delta_n},
\end{equation*}
where $\widetilde{C} = d/\alpha + 2M$ is a constant.
It follows that
\[ (L_n M_{n}k_{n})^{d^2}p_n^{k_n} \leqslant \big( \widetilde{C} \big)^{d^2} \delta_n^{-d^2} k_n^{d^2} p_n^{k_n} \leqslant \frac{\big( \widetilde{C} \big)^{d^2}}{k_n}, \]
where the last inequality follows from (\ref{eq:sequence-p-n}).
By (\ref{eq:condition-A-n-2}) and (\ref{eq:E-mu-V-n-upper-bound}), we conclude that \[ \lim_{n \to \f} \E(\mu\circ\mathcal{V}_n) = 0,\] as desired.
\end{proof}

Now, by applying Markov's inequality we can prove Proposition \ref{prop:main}.

\begin{proof}[Proof of Proposition \ref{prop:main}]
  Note first that the random variables $\mu\circ\cale_n$ and $\mu\circ\mathcal{V}_n$ are always in the range $[0,1]$.
  Fix $k \in \N$.
  By Markov's inequality, we have
  \[ \PP_n \Big\{ \boldsymbol{\omega} \in \Omega_n: 1-\mu\circ\cale_n(\boldsymbol{\omega}) \geqslant \frac{1}{k} \Big\} \leqslant k\big( 1- \E(\mu\circ\cale_n) \big), \]
  and \[ \PP_n \Big\{ \boldsymbol{\omega} \in \Omega_n: \mu\circ\mathcal{V}_n(\boldsymbol{\omega}) \geqslant \frac{1}{k} \Big\} \leqslant k \cdot \E(\mu\circ\mathcal{V}_n). \]
  By Lemma \ref{lemma:limit-E-n-V-n}, we can choose a large enough $n=n(k) \in \N$ such that
  \[ \PP_n \Big\{ \boldsymbol{\omega} \in \Omega_n: 1-\mu\circ\cale_n(\boldsymbol{\omega}) \geqslant \frac{1}{k} \Big\} < \frac{1}{4} \quad\text{and}\quad \PP_n \Big\{ \boldsymbol{\omega} \in \Omega_n: \mu\circ\mathcal{V}_n(\boldsymbol{\omega}) \geqslant \frac{1}{k} \Big\} < \frac{1}{4}. \]
  That is,
  \[ \PP_n \Big\{ \boldsymbol{\omega} \in \Omega_n: \mu\circ\cale_n(\boldsymbol{\omega}) > 1-\frac{1}{k} \Big\} > \frac{3}{4} \quad\text{and}\quad \PP_n \Big\{ \boldsymbol{\omega} \in \Omega_n: \mu\circ\mathcal{V}_n(\boldsymbol{\omega}) < \frac{1}{k} \Big\} > \frac{3}{4}. \]
  Thus, there exists $\boldsymbol{\omega} \in \Omega_n$ such that \[ \mu\circ\cale_n(\boldsymbol{\omega}) > 1-\frac{1}{k} \quad\text{and}\quad \mu\circ\mathcal{V}_n(\boldsymbol{\omega}) < \frac{1}{k}. \]
  Let $E_k = \cale_n(\boldsymbol{\omega})$.
  Then we have $\mu(E_k) > 1-1/k$ .
  By (\ref{eq:def-V-n-omega}), we clearly have $V_k=\big\{ \boldsymbol{x}\in [0,1]^d:\; \text{there exists}\; T\in\cals_{\alpha}^{1/\alpha}\; \text{such that}\; TA + \x \subseteq E_k \big\} \subseteq \mathcal{V}_n(\boldsymbol{\omega})$.
  So we have $\mu^{*}(V_k) \le \mu(\mathcal{V}_n(\boldsymbol{\omega}) )< 1/k$, where $\mu^{*}$ denote the Lebesgue outer measure.
  The sequence $\{E_k\}_{k=1}^\f$ is desired.
  The proof is completed.
\end{proof}

\section*{Acknowledgements}
W.~Li was supported by NSFC No.~12471085 and Science and Technology Commission of Shanghai Municipality (STCSM) No. 22DZ2229014.
Z.~Wang was supported by NSFC No.~12501110 and the China Postdoctoral Science Foundation No.~2024M763857.

\bibliographystyle{abbrv}
\bibliography{ErdosSimilarity}

\end{document}